\documentclass[11pt]{article}%
\usepackage{amsmath}
\usepackage{amsfonts}
\usepackage{amssymb}
\usepackage{graphicx}
\usepackage{hyperref}%
\newtheorem{theorem}{Theorem}[section]

\newtheorem{corollary}[theorem]{Corollary}

\newtheorem{example}[theorem]{Example}

\newtheorem{lemma}[theorem]{Lemma}

\newtheorem{proposition}[theorem]{Proposition}
\newtheorem{remark}[theorem]{Remark}

\newenvironment{proof}[1][Proof]{\noindent\textbf{#1.} }{\ \rule{0.5em}{0.5em}}
\begin{document}

\title{Representation of strongly truncated Riesz spaces}
\author{Karim Boulabiar\thanks{Corresponding author:
\texttt{karim.boulabiar@fst.utm.tn}}\quad and\quad Rawaa Hajji\medskip\\{\small Laboratoire de Recherche LATAO}\\{\small D\'{e}partement de Mathematiques, Facult\'{e} des Sciences de Tunis}\\{\small Universit\'{e} de Tunis El Manar, 2092, El Manar, Tunisia}}
\date{\textit{Dedicated to the memory of Professor Wim Luxemburg}}
\maketitle

\begin{abstract}
Following a recent idea by Ball, we introduce the notion of strongly truncated
Riesz space with a suitable spectrum. We prove that, under an extra
Archimedean type condition, any strongly truncated Riesz space is isomorphic
to a uniformly dense Riesz subspace of a $C_{0}\left(  X\right)  $-space. This
turns out to be a direct generalization of the classical Kakutani
Representation Theorem on Archimedean Riesz spaces with strong unit. Another
representation theorem on normed Riesz spaces, due to Fremlin, will be
obtained as a consequence of our main result.

\end{abstract}



\section{Introduction}

One of major contributions by Stone to Measure Theory is a representation
pre-intergrals $I\left(  f\right)  $ as intergrals $\int fd\mu$ for some
measure $\mu$ (see \cite{S48}). He found in particular that the natural
framework to get such a representation is the function Riesz spaces $E$ closed
under \textsl{truncation} with the constant function $1$, i.e.,%
\[
1\wedge f\in E\text{, for all }f\in E.
\]
Since then, such spaces, called by now \textsl{Stone vector lattices}, were
found to be fundamental to analysis, and their importance has steadily grown
\cite{AC07,D04,HP84,K97}. For instance, Lebesgue integration extends quite
smoothly to any Stone vector lattice, and function Riesz spaces lacking the
above \textsl{Stone condition} may have pre-integrals which cannot be
represented by any measure (see \cite{JF87}). In \cite{F74,F06}, Fremlin calls
Stone vector lattices of functions \textsl{truncated Riesz spaces}. In the
present paper, we shall rather adopt the Fremlin terminology. That's what Ball
did in his papers \cite{B14,B14-0} when he gave an appropriate axiomatization
of this concept. Rephrasing his first Axiom, we call a \textsl{truncation} on
an arbitrary Riesz space $E$ any unary operation $\ast$ on $E^{+}$ such that%
\[
f^{\ast}\wedge g=f\wedge g^{\ast}\text{, for all }f,g\in E^{+}.
\]
A Riesz space (also called a vector lattice) $E$ along with a truncation is
called a \textsl{truncated Riesz space}. A positive element $f$ in a truncated
Riesz space $E$ is said to be $^{\ast}$-\textsl{infinitely small} if%
\[
\left(  \varepsilon f\right)  ^{\ast}=\varepsilon f\text{, for all
}\varepsilon\in\left(  0,\infty\right)  .
\]
Furthermore, if $f\in E^{+}$ and $f^{\ast}=0$ imply $f=0$, we call $E$ a
\textsl{weakly truncated Riesz space}. Using a different terminology, Ball
proved that for any weakly truncated Riesz space $E$ with no nonzero $^{\ast}%
$-infinitely small elements, there exists a compact Hausdorff space $K$ such
that $E$ can be represented as a truncated Riesz space of continuous real
extended functions on $K$ (see \cite{L79} for continuous real extended
functions). Actually, this is a direct generalization of the classical Yosida
Representation Theorem for Archimedean Riesz spaces with weak unit
\cite{LZ71}. The Ball's result prompted us to investigate the extent to which
the classical Kakutani Representation Theorem for Archimedean Riesz spaces
with strong unit can be generalized to a wider class of truncated Riesz
spaces. Let's figure out how the problem has been addressed.

A truncation $\ast$ on a Riesz space $E$ is said to be \textsl{strong }if, for
every $f\in E^{+}$, the equality $\left(  \varepsilon f\right)  ^{\ast
}=\varepsilon f$ holds for some $\varepsilon\in\left(  0,\infty\right)  $. A
Riesz space $E$ with a strong truncation is called a \textsl{strongly
truncated Riesz space}. The set of all real-valued Riesz homomorphisms $u$ on
the truncated Riesz space $E$ such that%
\[
u\left(  f^{\ast}\right)  =\min\left\{  1,u\left(  f\right)  \right\}  \text{,
for all }f\in E^{+}%
\]
is denoted by $\eta E$ and called the \textsl{spectrum} of $E$. We show that
$\eta E$, under its topology inherited from the product topology on
$\mathbb{R}^{E}$, is a locally compact Hausdorff space. Moreover, if $E$ is a
strongly truncated Riesz with no nonzero $^{\ast}$-infinitely small elements,
then $\eta E$ is large enough to separate the points of $E$ and allow
representation by continuous functions. Relying on a lattice version of the
Stone-Weierstrass theorem for locally compact Hausdorff spaces (which could
not be found explicitly in the literature), we prove that any strongly
truncated Riesz space $E$ with no nonzero $^{\ast}$-infinitely small elements
is Riesz isomorphic to a uniformly dense truncated Riesz subspace of
$C_{0}\left(  \eta E\right)  $. Here, $C_{0}\left(  \eta E\right)  $ denotes
the Banach lattice of all continuous real-valued functions on $\eta E$
vanishing at infinity. We obtain the classical Kakutani Representation Theorem
as a special case of our result. Indeed, if $E$ is an Archimedean Riesz space
with a strong unit $e>0$, then $E$ is a strongly truncated Riesz space with no
nonzero $^{\ast}$-infinitely elements under the truncation given by%
\[
f^{\ast}=e\wedge f\text{, for all }f\in E^{+}.
\]
We show that, in this situation, $\eta E$ is a compact Hausdorff space and so
$C_{0}\left(  \eta E\right)  $ coincides with the Banach lattice $C\left(
\eta E\right)  $ of all continuous real-valued functions on $\eta E$. In
summary, $E$ has a uniformly dense copy in $C\left(  \eta E\right)  $. Our
central result will also turn out to be a considerable generalization of a
less known representation theorem due to Fremlin (see \cite[83L (d)]{F74}).
Some details seem in order.

Let $E$ be a Riesz space with a Fatou $M$-norm $\left\Vert .\right\Vert $ and
assume that the supremum%
\[
\sup\left\{  g\in E^{+}:g\leq f\text{ and }\left\Vert g\right\Vert \leq
r\right\}
\]
exists in $E$ for every $f\in E^{+}$ and $r\in\left(  0,\infty\right)  $.
Fremlin proved that $E$ is isomorphic, as a normed Riesz space, to a truncated
Riesz subspace of the Riesz space $\ell^{\infty}\left(  X\right)  $ of all
bounded real-valued functions on some nonvoid set $X$. As a matter of fact, we
shall use our main theorem to improve this result by showing that $E$ is
isomorphic, as a normed Riesz space, to a uniformly dense truncated Riesz
subspace of $C_{0}\left(  \eta E\right)  $.

Finally, we point out that in each section we summarize enough necessary
background material to keep this paper reasonably self contained. By the way,
the classical book \cite{LZ71} by Luxemburg and Zaanen is adopted as the
unique source of unexplained terminology and notation.

\section{Infinitely small elements with respect to a truncation}

Recall that a \textsl{truncation} on a Riesz space $E$ is a unary operation
$\ast$ on the positive cone $E^{+}$ of $E$ such that%
\[
f^{\ast}\wedge g=f\wedge g^{\ast}\text{, for all }f,g\in E^{+}.
\]
A Riesz space $E$ along with a truncation $\ast$ is called a \textsl{truncated
Riesz space.} The truncation, on any truncated Riesz space will be denoted by
$\ast$. The set of all fixed points of the truncation on a Riesz space $E$ is
denoted by $P_{\ast}\left(  E\right)  $, i.e.,%
\[
P_{\ast}\left(  E\right)  =\left\{  f\in E^{+}:f^{\ast}=f\right\}  .
\]
We gather some elementary properties in the next lemma. Some of them can be
found in \cite{BE17}. We give the detailed proofs for the sake of convenience.

\begin{lemma}
\label{elem}Let $E$ be a truncated Riesz space and $f,g\in E^{+}$. Then the
following hold.

\begin{enumerate}
\item[\emph{(i)}] $f^{\ast}\leq f$ and $f^{\ast}\in P_{\ast}\left(  E\right)
$.

\item[\emph{(ii)}] $f\leq g$ implies $f^{\ast}\leq g^{\ast}$.

\item[\emph{(iii)}] If $f\leq g$ and $g\in P_{\ast}\left(  E\right)  $, then
$f\in P_{\ast}\left(  E\right)  $.

\item[\emph{(iv)}] $\left(  f\wedge g\right)  ^{\ast}=f^{\ast}\wedge g^{\ast}$
and $\left(  f\vee g\right)  ^{\ast}=f^{\ast}\vee g^{\ast}$ \emph{(}In
particular, $P_{\ast}\left(  E\right)  $ is a lattice\emph{)}.

\item[\emph{(v)}] $\left\vert f^{\ast}-g^{\ast}\right\vert \leq\left\vert
f-g\right\vert ^{\ast}$.
\end{enumerate}
\end{lemma}

\begin{proof}
$\mathrm{(i)}$ We have%
\[
f^{\ast}=f^{\ast}\wedge f^{\ast}=f\wedge\left(  f^{\ast}\right)  ^{\ast}\leq
f.
\]
In particular, $\left(  f^{\ast}\right)  ^{\ast}\leq f^{\ast}\leq f$ and so%
\[
\left(  f^{\ast}\right)  ^{\ast}=\left(  f^{\ast}\right)  ^{\ast}\wedge
f=f^{\ast}\wedge f^{\ast}=f^{\ast}.
\]
This shows $\mathrm{(i)}$.

$\mathrm{(ii)}$ Since $f^{\ast}\leq f\leq g$, we get%
\[
f^{\ast}=f^{\ast}\wedge g=f\wedge g^{\ast}\leq g^{\ast},
\]
which gives the desired inequality.

$\mathrm{(iii)}$ As in the proof of $\mathrm{(ii)}$, from $f^{\ast}\leq f\leq
g$ it follows that%
\[
f^{\ast}=f^{\ast}\wedge g=f\wedge g^{\ast}=f.
\]
This means $f$ is a fixed point of $\ast$.

$\mathrm{(iv)}$ Using $\mathrm{(ii)}$, we have $\left(  f\wedge g\right)
^{\ast}\leq f^{\ast}$ and $\left(  f\wedge g\right)  ^{\ast}\leq g^{\ast}$.
Thus,%
\[
\left(  f\wedge g\right)  ^{\ast}\leq f^{\ast}\wedge g^{\ast}=f\wedge g\wedge
g^{\ast}=\left(  f\wedge g\right)  ^{\ast}\wedge g\leq\left(  f\wedge
g\right)  ^{\ast}.
\]
This shows the first equality. Now, by $\mathrm{(i)}$, we have%
\begin{align*}
f^{\ast}\vee g^{\ast}  &  =\left(  f\vee g\right)  \wedge\left(  f^{\ast}\vee
g^{\ast}\right) \\
&  =\left(  \left(  f\vee g\right)  \wedge f^{\ast}\right)  \vee\left(
\left(  f\vee g\right)  \wedge g^{\ast}\right) \\
&  =\left(  \left(  f\vee g\right)  ^{\ast}\wedge f\right)  \vee\left(
\left(  f\vee g\right)  ^{\ast}\wedge g\right) \\
&  =\left(  f\vee g\right)  ^{\ast}\wedge\left(  f\vee g\right)  =\left(
f\vee g\right)  ^{\ast},
\end{align*}
and the second equality follows.

$\mathrm{(v)}$ From $\mathrm{(i)}$ it follows that $f^{\ast}\leq f\leq f+g$
and $g^{\ast}\leq g\leq f+g$. Hence, using the classical Birkhoff's Inequality
(see \cite[Theorem 1.9. (b)]{AB06}), we obtain
\begin{align*}
\left\vert f^{\ast}-g^{\ast}\right\vert  &  =\left\vert f^{\ast}\wedge\left(
f+g\right)  -g^{\ast}\wedge\left(  f+g\right)  \right\vert =\left\vert
f\wedge\left(  f+g\right)  ^{\ast}-g\wedge\left(  f+g\right)  ^{\ast
}\right\vert \\
&  \leq\left\vert f-g\right\vert \wedge\left(  f+g\right)  ^{\ast}=\left\vert
f-g\right\vert ^{\ast}\wedge\left(  f+g\right)  \leq\left\vert f-g\right\vert
^{\ast}.
\end{align*}
The proof of the lemma is now complete.
\end{proof}

An element $f$ in a truncated Riesz space $E$ is said to be $^{\ast}%
$\textsl{-infinitely small} if%
\[
\varepsilon\left\vert f\right\vert \in P_{\ast}\left(  E\right)  \text{, for
all }\varepsilon\in\left(  0,\infty\right)  .
\]
The set $E_{\ast}$ of all $^{\ast}$-infinitely small elements in $E$ enjoys an
interesting algebraic property.

\begin{lemma}
Let $E$ be a truncated Riesz space. Then $E_{\ast}$ is an ideal in $E$.
\end{lemma}

\begin{proof}
Let $f,g\in E_{\ast}$ and $\alpha\in\mathbb{R}$. Pick $\varepsilon\in\left(
0,\infty\right)  $ and observe that%
\begin{align*}
\varepsilon\left\vert f+\alpha g\right\vert  &  \leq\left(  1+\left\vert
\alpha\right\vert \right)  \left(  \frac{1}{1+\left\vert \alpha\right\vert
}\varepsilon\left\vert f\right\vert +\frac{\left\vert \alpha\right\vert
}{1+\left\vert \alpha\right\vert }\varepsilon\left\vert g\right\vert \right)
\\
&  \leq\left(  1+\left\vert \alpha\right\vert \right)  \varepsilon\left\vert
f\right\vert \vee\left(  1+\left\vert \alpha\right\vert \right)
\varepsilon\left\vert g\right\vert .
\end{align*}
However, using Lemma \ref{elem} $\mathrm{(iv)}$, we find%
\begin{align*}
\left(  \left(  1+\left\vert \alpha\right\vert \right)  \varepsilon\left\vert
f\right\vert \vee\left(  1+\left\vert \alpha\right\vert \right)
\varepsilon\left\vert g\right\vert \right)  ^{\ast}  &  =\left(  \left(
1+\left\vert \alpha\right\vert \right)  \varepsilon\left\vert f\right\vert
\right)  ^{\ast}\vee\left(  \left(  1+\left\vert \alpha\right\vert \right)
\varepsilon\left\vert g\right\vert \right)  ^{\ast}\\
&  =\left(  1+\left\vert \alpha\right\vert \right)  \varepsilon\left\vert
f\right\vert \vee\left(  1+\left\vert \alpha\right\vert \right)
\varepsilon\left\vert g\right\vert .
\end{align*}
It follows that%
\[
\left(  1+\left\vert \alpha\right\vert \right)  \varepsilon\left\vert
f\right\vert \vee\left(  1+\left\vert \alpha\right\vert \right)
\varepsilon\left\vert g\right\vert \in P_{\ast}\left(  E\right)  \mathrm{.}%
\]
and so, by Lemma \ref{elem} $\mathrm{(iii)}$, $\varepsilon\left\vert f+\alpha
g\right\vert \in P_{\ast}\left(  E\right)  $. But then $\left\vert f+\alpha
g\right\vert \in E_{\ast}$ because $\varepsilon$ is arbitrary in $\left(
0,\infty\right)  $. Accordingly, $E_{\ast}$ is a vector subspace of $E$. Now,
let $f,g\in E$ such that $\left\vert f\right\vert \leq\left\vert g\right\vert
$ and $g\in E_{\ast}$. If $\varepsilon\in\left(  0,\infty\right)  $ then
$\varepsilon\left\vert f\right\vert \leq\varepsilon\left\vert g\right\vert \in
P_{\ast}\left(  E\right)  $. Using once again \ref{elem} $\mathrm{(iii)}$, we
conclude that $E_{\ast}$ is a solid in $E$ and the lemma follows.
\end{proof}

Throughout this paper, if $I$ is an ideal in a Riesz space $E$, then the
equivalence class in the quotient Riesz space $E/I$ of an element $f\in E$
will be denoted by $I\left(  f\right)  $. In other words, the canonical
surjection from $E\ $onto $E/I$, which is a Riesz homomorphism, will be
denoted by $I$ as well (we refer the reader to \cite[Section 18]{LZ71} for
quotient Riesz spaces).

In what follows, we show that the quotient Riesz space $E/E_{\ast}$ can be
equipped with a truncation in a natural way.

\begin{proposition}
\label{quotient}Let $E$ be a truncated Riesz space. Then, the unary operation
$E_{\ast}\left(  f\right)  \rightarrow E_{\ast}\left(  f\right)  ^{\ast}$ on
$\left(  E/E_{\ast}\right)  ^{+}$ given by%
\[
E_{\ast}\left(  f\right)  ^{\ast}=E_{\ast}\left(  f^{\ast}\right)  \text{, for
all }f\in E^{+}\text{,}%
\]
is a truncation on $E/E_{\ast}$. Moreover, the truncated Riesz space
$E/E_{\ast}$ has no nonzero $^{\ast}$-infinitely small elements.
\end{proposition}

\begin{proof}
Let $f,g\in E$ such that $E_{\ast}\left(  f\right)  =E_{\ast}\left(  g\right)
$. Hence, $f-g\in E_{\ast}$ and so $\left\vert f-g\right\vert ^{\ast}\in
E_{\ast}$ because $E_{\ast}$ is an ideal and $\left\vert f-g\right\vert
^{\ast}\leq\left\vert f-g\right\vert $. Using Lemma \ref{elem} $\mathrm{(v)}$,
we get $\left\vert f^{\ast}-g^{\ast}\right\vert \in E_{\ast}$ and thus
$E_{\ast}\left(  f^{\ast}\right)  =E_{\ast}\left(  g^{\ast}\right)  $.
Moreover, if $f,g\in E^{+}$ then%
\[
E_{\ast}\left(  f^{\ast}\right)  \wedge E_{\ast}\left(  g\right)  =E_{\ast
}\left(  f^{\ast}\wedge g\right)  =E_{\ast}\left(  f\wedge g^{\ast}\right)
=E_{\ast}\left(  f\right)  \wedge E_{\ast}\left(  g^{\ast}\right)  .
\]
It follows that the equality%
\[
E_{\ast}\left(  f\right)  ^{\ast}=E_{\ast}\left(  f^{\ast}\right)  \text{, for
all }f\in E^{+}%
\]
defines a truncation on $E/E_{\ast}$. Furthermore, pick $f\in E^{+}$ and
assume that%
\[
\varepsilon E_{\ast}\left(  f\right)  \in P_{\ast}\left(  E/E_{\ast}\right)
\text{, for all }\varepsilon\in\left(  0,\infty\right)  .
\]
For every $\varepsilon\in\left(  0,\infty\right)  $, we can write%
\[
E_{\ast}\left(  \varepsilon f\right)  =\varepsilon E_{\ast}\left(  f\right)
=\left(  \varepsilon E_{\ast}\left(  f\right)  \right)  ^{\ast}=\left(
E_{\ast}\left(  \varepsilon f\right)  \right)  ^{\ast}=E_{\ast}\left(  \left(
\varepsilon f\right)  ^{\ast}\right)  .
\]
Therefore,%
\[
\varepsilon f-\left(  \varepsilon f\right)  ^{\ast}\in E_{\ast}\text{, for all
}\varepsilon\in\left(  0,\infty\right)  .
\]
In particular,%
\[
\varepsilon f-\left(  \varepsilon f\right)  ^{\ast}\in P_{\ast}\left(
E\right)  \text{, for all }\varepsilon\in\left(  0,\infty\right)
\]
Accordingly, if $\varepsilon\in\left(  0,\infty\right)  $ then%
\begin{align*}
\varepsilon f  &  =\varepsilon f-\left(  \varepsilon f\right)  ^{\ast}+\left(
\varepsilon f\right)  ^{\ast}=\left(  \varepsilon f-\left(  \varepsilon
f\right)  ^{\ast}\right)  ^{\ast}+\left(  \varepsilon f\right)  ^{\ast}\\
&  \leq2\left[  \left(  \varepsilon f-\varepsilon f^{\ast}\right)  ^{\ast}%
\vee\left(  \varepsilon f\right)  ^{\ast}\right]  =2\left[  \left(
\varepsilon f-\varepsilon f^{\ast}\right)  \vee\varepsilon f\right]  ^{\ast}.
\end{align*}
It follows that%
\[
0\leq\frac{\varepsilon}{2}f\leq\left[  \left(  \varepsilon f-\varepsilon
f^{\ast}\right)  \vee\varepsilon f\right]  ^{\ast}\in P_{\ast}\left(
E\right)  \text{, for all }\varepsilon\in\left(  0,\infty\right)  .
\]
But then $f\in E_{\ast}$, so $E_{\ast}\left(  f\right)  =0$. Consequently, the
truncated Riesz space $E/E_{\ast}$ has no nonzero $^{\ast}$-infinitely small
elements, as desired.
\end{proof}

\section{Spectrum of a truncated Riesz space}

Let $E,F$ be two truncated Riesz spaces. A Riesz homomorphism $T:E\rightarrow
F$ is called a $^{\ast}$-\textsl{homomorphism} if $T$ preserves truncations,
i.e.,%
\[
T\left(  f^{\ast}\right)  =T\left(  f\right)  ^{\ast}\text{, for all }f\in
E^{+}.
\]
For instance, it follows directly from Proposition \ref{quotient} that the
canonical surjection $E_{\ast}:E\rightarrow E/E_{\ast}$ is a $^{\ast}%
$-homomorphism. Clearly, if $T:E\rightarrow F$ is a bijective $^{\ast}%
$-homomorphism, then its inverse $T^{-1}:F\rightarrow E$ is a $^{\ast}%
$-homomorphism. In this situation, $T$ is called a $^{\ast}$%
\textsl{-isomorphism} and $E,F$ are said to be $^{\ast}$\textsl{-isomorphic}.
Now, we define the \textsl{canonical truncation} $\ast$ on the Riesz space
$\mathbb{R}$ of all real numbers by putting%
\[
r^{\ast}=\min\left\{  1,r\right\}  \text{, for all }r\in\left[  0,\infty
\right)  .
\]
The set of all nonzero $^{\ast}$-homomorphisms from $E$ onto $\mathbb{R}$ is
called the \textsl{spectrum} of $E$ and denoted by $\eta E$.

\begin{quote}
\textsl{Beginning with the next lines, }$\eta E$ \textsl{will be equipped with
the topology inherited from the product topology on the Tychonoff space}
$\mathbb{R}^{E}$.
\end{quote}

\noindent Recall here that the product topology on $\mathbb{R}^{E}$ is the
coarsest topology on $\mathbb{R}^{E}$ for which all projections $\pi_{f}$
$\left(  f\in E\right)  $ are continuous, where $\pi_{f}:\mathbb{R}%
^{E}\rightarrow\mathbb{R}$ is defined by%
\[
\pi_{f}\left(  \phi\right)  =\phi\left(  f\right)  \text{, for all }\phi
\in\mathbb{R}^{E}.
\]
A truncation on the Riesz space $E$ is said to be \textsl{strong} if, for each
$f\in E$, there exists $\varepsilon\in\left(  0,\infty\right)  $ such that
$\varepsilon f\in P_{\ast}\left(  E\right)  $. A Riesz space along with a
strong truncation is called a \textsl{strongly truncated Riesz space. }It
turns out that the spectrum of a strongly truncated Riesz space has an
interesting topological property.

\begin{lemma}
\label{loc}The spectrum $\eta E$ of a strongly truncated Riesz space $E$ is
locally compact.
\end{lemma}

\begin{proof}
If $f\in E$, then there exists $\varepsilon_{f}\in\left(  0,\infty\right)  $
such that $\varepsilon_{f}\left\vert f\right\vert \in P_{\ast}\left(
E\right)  $. So, for each $u\in\eta E$, we have%
\[
\left\vert u\left(  f\right)  \right\vert =\frac{1}{\varepsilon_{f}}u\left(
\left\vert \varepsilon_{f}f\right\vert ^{\ast}\right)  =\frac{1}%
{\varepsilon_{f}}\min\left\{  1,u\left(  \left\vert \varepsilon_{f}%
f\right\vert \right)  \right\}  \leq\frac{1}{\varepsilon_{f}}.
\]
Accordingly,%
\[
\eta E\cup\left\{  0\right\}  \subset%
%TCIMACRO{\dprod \limits_{f\in E}}%
%BeginExpansion
{\displaystyle\prod\limits_{f\in E}}
%EndExpansion
\left[  -\frac{1}{\varepsilon_{f}},\frac{1}{\varepsilon_{f}}\right]  .
\]
On the other hand, it takes no more than a moment's thought to see that $\eta
E\cup\left\{  0\right\}  $ is the intersection of the closed sets%
\[%
\begin{tabular}
[c]{c}%
$%
%TCIMACRO{\dbigcap \limits_{f,g\in E,\lambda\in\mathbb{R}}}%
%BeginExpansion
{\displaystyle\bigcap\limits_{f,g\in E,\lambda\in\mathbb{R}}}
%EndExpansion
\left(  \pi_{f+\lambda g}-\pi_{f}-\lambda\pi_{g}\right)  ^{-1}\left(  \left\{
0\right\}  \right)  \text{,}$\\
$%
%TCIMACRO{\dbigcap \limits_{f\in E}}%
%BeginExpansion
{\displaystyle\bigcap\limits_{f\in E}}
%EndExpansion
\left(  \left\vert \pi_{f}\right\vert -\pi_{\left\vert f\right\vert }\right)
^{-1}\left(  \left\{  0\right\}  \right)  \text{, and }%
%TCIMACRO{\dbigcap \limits_{f\in E_{+}}}%
%BeginExpansion
{\displaystyle\bigcap\limits_{f\in E_{+}}}
%EndExpansion
\left(  \pi_{f^{\ast}}-1\wedge\pi_{f}\right)  ^{-1}\left(  \left\{  0\right\}
\right)  .$%
\end{tabular}
\
\]
It follows that $\eta E\cup\left\{  0\right\}  $ is again a closed set in
$\mathbb{R}^{E}$. In summary, $\eta E\cup\left\{  0\right\}  $ is a closed
subset of a compact set and so it is compact. But then $\eta E$ is locally
compact since it is an open subset in a compact set.
\end{proof}

The following simple lemma is needed to establish the next theorem.

\begin{lemma}
\label{weak}Let $E$ be a strongly truncated Riesz space. If $f\in E^{+}$ and
$f^{\ast}=0$ then $f=0$.
\end{lemma}

\begin{proof}
Choose $\varepsilon\in\left(  0,\infty\right)  $ such that $\varepsilon f\in
P_{\ast}\left(  f\right)  $. If $\varepsilon\leq1$ then%
\[
0\leq\varepsilon f=\left(  \varepsilon f\right)  ^{\ast}\leq f^{\ast}=0
\]
(where we use Lemma \ref{elem} $\mathrm{(ii)}$). Suppose now that
$\varepsilon>1$. We get%
\[
0\leq f=f\wedge\varepsilon f=f\wedge\left(  \varepsilon f\right)  ^{\ast
}=f^{\ast}\wedge\varepsilon f=0.
\]
This completes the proof of the lemma.
\end{proof}

In fact, Lemma \ref{weak} tells us that any strong truncation on Riesz space
is a weak truncation.

Now, the kernel of any $^{\ast}$-homomorphism $u\in\eta E$ is denoted by $\ker
u$. The following theorem will turn out to be crucial for later purposes.
Actually, we are indebted to professor Richard Ball for his significant help.
Indeed, trying to prove the result, we ran into a serious problem and it was
only through him that things were done.

\begin{theorem}
\label{spectrum}Let $E$ be a strongly truncated Riesz space. Then%
\[%
%TCIMACRO{\dbigcap \limits_{u\in\eta E}}%
%BeginExpansion
{\displaystyle\bigcap\limits_{u\in\eta E}}
%EndExpansion
\ker u=E_{\ast}.
\]

\end{theorem}

\begin{proof}
First, we assume that $E_{\ast}=\left\{  0\right\}  $. We shall prove that if
$0<f\in E$, then $u\left(  f\right)  >0$ for some $u\in\eta E$. Choose
$\varepsilon\in\left(  0,\infty\right)  $ for which $\varepsilon f\notin
P_{\ast}\left(  E\right)  $. By replacing $f$ by $\varepsilon f$ if needed, we
may suppose that $f^{\ast}<f$, so $\left(  f-f^{\ast}\right)  ^{\ast}>0$ (see
Lemma \ref{weak}). Let $P$ be a prime ideal in $E$ such that $\left(
f-f^{\ast}\right)  ^{\ast}\notin P$ (such an ideal exists by \cite[Theorem
33.5]{LZ71}). If $g\in E^{+}$ then%
\[
0<\left(  f-f^{\ast}\right)  ^{\ast}\leq f-f^{\ast}\leq f-\left(  f^{\ast
}\wedge g\right)  =f-\left(  f\wedge g^{\ast}\right)  =\left(  f-g^{\ast
}\right)  ^{+}.
\]
Accordingly, $\left(  f-g^{\ast}\right)  ^{+}\notin P$ and, as $P$ is prime,%
\[
g^{\ast}-g^{\ast}\wedge f=\left(  g^{\ast}-f\right)  ^{+}=\left(  f-g^{\ast
}\right)  ^{-}\in P\text{.}%
\]
On the other hand, by Theorem 33.5 in \cite{LZ71}, there exists a prime ideal
$Q$, containing $P$, which is maximal with respect to the property of not
containing $f^{\ast}$. Of course, such a prime ideal $Q$ indeed exists because
$f^{\ast}\notin P$. Then, from%
\[
g^{\ast}-g\wedge f^{\ast}\in P\subset Q\text{, for all }g\in E^{+},
\]
it follows that the (in)equalities%
\[
Q\left(  g^{\ast}\right)  =Q\left(  g\wedge f^{\ast}\right)  =Q\left(
g\right)  \wedge Q\left(  f^{\ast}\right)  \leq Q\left(  f^{\ast}\right)
\]
hold in the quotient Riesz space $E/Q$ for all $g\in E^{+}$. Let $g\in E^{+}$
and $\varepsilon\in\left(  0,\infty\right)  $ such that $\varepsilon g\in
P_{\ast}\left(  E\right)  $. So,%
\[
Q\left(  g\right)  =\frac{1}{\varepsilon}Q\left(  \varepsilon g\right)
=\frac{1}{\varepsilon}Q\left(  \left(  \varepsilon g\right)  ^{\ast}\right)
\leq\frac{1}{\varepsilon}Q\left(  f^{\ast}\right)  .
\]
We derive that $Q\left(  f^{\ast}\right)  $ is a strong unit in $E/Q$. We
claim that $E/Q$ is Riesz isomorphic to $\mathbb{R}$. To this end, let $I$ be
a proper ideal in $E/Q$ and set%
\[
Q^{-1}\left(  I\right)  =\left\{  g\in E:Q\left(  g\right)  \in I\right\}  .
\]
Then, $Q^{-1}\left(  I\right)  $ is an ideal in $E$ containing $Q$. If
$f^{\ast}\in Q^{-1}\left(  I\right)  $ then $Q\left(  f^{\ast}\right)  \in I$
and so $I=E/Q$ (as $Q\left(  f^{\ast}\right)  $ is a strong unit in
$E/E_{\ast}$), a contradiction. Accordingly, $f^{\ast}\notin Q^{-1}\left(
I\right)  $ and so, by maximality, $Q^{-1}\left(  I\right)  =Q$. We derive
straightforwardly that $I=\left\{  0\right\}  $ and, in view of \cite[Theorem
27.1]{LZ71}, there exists a Riesz isomorphism $\varphi:E/Q\rightarrow
\mathbb{R}$ with $\varphi\left(  Q\left(  f^{\ast}\right)  \right)  =1$. Put
$u=\varphi\circ Q$ and notice that $u$ is a Riesz homomorphism. Moreover, if
$g\in E^{+}$ then%
\begin{align*}
u\left(  g^{\ast}\right)   &  =\left(  \varphi\circ Q\right)  \left(  g^{\ast
}\right)  =\varphi\left(  Q\left(  g^{\ast}\right)  \right)  =\varphi\left(
Q\left(  g\right)  \wedge Q\left(  f^{\ast}\right)  \right) \\
&  =\varphi\left(  Q\left(  g\right)  \right)  \wedge\varphi\left(  Q\left(
f^{\ast}\right)  \right)  =\min\left\{  1,u\left(  g\right)  \right\}  .
\end{align*}
This yields that $u\in\eta E$. Furthermore,%
\[
u\left(  f\right)  \geq u\left(  f^{\ast}\right)  =\varphi\left(  Q\left(
f^{\ast}\right)  \right)  =1>0,
\]
Consequently,%
\[%
%TCIMACRO{\dbigcap \limits_{u\in\eta E}}%
%BeginExpansion
{\displaystyle\bigcap\limits_{u\in\eta E}}
%EndExpansion
\ker u=\left\{  0\right\}  =E_{\ast}.
\]
Let's discuss the general case. Pick $f\in E$ such that%
\[
u\left(  f\right)  =0\text{, for all }u\in\eta E.
\]
We claim that $f\in E_{\ast}$. To this end, choose $\phi\in\eta\left(
E/E_{\ast}\right)  $ and observe that $\phi\circ E_{\ast}\in\eta E$ (where we
use Proposition \ref{quotient}). This means that%
\[
\phi\left(  E_{\ast}\left(  f\right)  \right)  =\left(  \phi\circ E_{\ast
}\right)  \left(  f\right)  =0.
\]
By the first case, $E_{\ast}\left(  f\right)  =0$ and thus $f\in E_{\ast}$. We
get the inclusion%
\[%
%TCIMACRO{\dbigcap \limits_{u\in\eta E}}%
%BeginExpansion
{\displaystyle\bigcap\limits_{u\in\eta E}}
%EndExpansion
\ker u\subset E_{\ast}.
\]
The converse inclusion is routine.
\end{proof}

\section{The representation theorem}

This section contains the central result of this paper, viz., a representation
theorem for strongly truncated Riesz spaces with no nonzero $^{\ast}%
$-infinitely small elements. As most classical theorems of representation, our
result is based upon a Stone-Weierstrass type theorem, which is presumably
well-known. Unfortunately, we have not been able to locate a precise reference
for it. We have therefore chosen to provide a detailed proof, which is an
adjustment of the proof of the \textquotedblleft algebra
version\textquotedblright\ of the theorem (see, for instance, Corollary 4.3.5
in \cite{P80}). In this regard, we need further prerequisites.

Let $X$ be a locally compact Hausdorff space $X$. The Riesz space of all
real-valued continuous functions on $X$ is denoted by $C\left(  X\right)  $,
as usual. A function $f\in C\left(  X\right)  $ is said to \textsl{vanish at
infinity} if, for every $\varepsilon\in\left(  0,\infty\right)  $, the set%
\[
K\left(  f,\varepsilon\right)  =\left\{  x\in X:\left\vert f\left(  x\right)
\right\vert \geq\varepsilon\right\}
\]
is compact. The collection $C_{0}\left(  X\right)  $ of such functions is a
Riesz subspace of $C\left(  X\right)  $. Actually, $C_{0}\left(  X\right)  $
is a Banach lattice (more precisely, an $AM$-space \cite{S74}) under the
uniform norm given by%
\[
\left\Vert f\right\Vert _{\infty}=\sup\left\{  \left\vert f\left(  x\right)
\right\vert :x\in X\right\}  \text{, for all }f\in C_{0}\left(  X\right)  .
\]
If $f\in C_{0}\left(  X\right)  $, then the real-valued function $f^{\infty}$
defined on the one-point compactification $X_{\infty}=X\cup\left\{
\infty\right\}  $ of $X$ by%
\[
f_{\infty}\left(  \infty\right)  =0\quad\text{and\quad}f_{\infty}\left(
x\right)  =f\left(  x\right)  \text{ if }x\in X
\]
is the unique extension of $f$ in $C\left(  X_{\infty}\right)  $ (see, e.g.,
\cite{E89}). Here too, $C\left(  X_{\infty}\right)  $ is endowed with its
uniform norm defined by%
\[
\left\Vert f\right\Vert _{\infty}=\sup\left\{  \left\vert f\left(  x\right)
\right\vert :x\in X_{\infty}\right\}  \text{, for all }f\in C\left(
X_{\infty}\right)  .
\]
It is an easy exercise to verify that the map $S:C_{0}\left(  X\right)
\rightarrow C\left(  X_{\infty}\right)  $ defined by%
\[
S\left(  f\right)  =f_{\infty}\text{, for all }f\in C_{0}\left(  X\right)
\]
is an isometry and, simultaneously, a Riesz isomorphism. In what follows,
$C_{0}\left(  X\right)  $ will be identified with the range of $S$, which is a
uniformly closed Riesz subspace of $C\left(  X_{\infty}\right)  $. On the
other hand, a subset $D$ of $C_{0}\left(  X\right)  $ is said to
\textsl{vanish nowhere} if for every $x\in X$, there is some $f\in D$ such
that $f\left(  x\right)  \neq0$. Moreover, $D$ is said to \textsl{separate the
points} of $X$ if for every $x,y\in X$ with $x\neq y$, we can find some $f\in
D$ such that $f\left(  x\right)  \neq f\left(  y\right)  $. Following Fremlin
in \cite{F74}, we call a \textsl{truncated Riesz subspace} of $C_{0}\left(
X\right)  $ any Riesz subspace $E$ of $C_{0}\left(  X\right)  $ for which%
\[
1\wedge f\in E\text{, for all }f\in E.
\]
We are in position now to prove the suitable version of the Stone-Weierstrass
theorem we were talking about.

\begin{lemma}
\label{SW}Let $X$ be a locally compact Hausdorff space and $E$ be a truncated
Riesz subspace of $C_{0}\left(  X\right)  $. Then $E$ is uniformly dense in
$C_{0}\left(  X\right)  $ if and only if $E$ vanishes nowhere and separates
the points of $X$.
\end{lemma}

\begin{proof}
We prove the `\textit{if}' part. So, assume that $E$ vanishes nowhere and
separates the points of $X$. It is an easy task to check that $E$ separates
the points of $X_{\infty}$. Consider the direct sum%
\[
E\oplus\mathbb{R}=\left\{  f+r:f\in E\text{ and }r\in\mathbb{R}\right\}  .
\]
Clearly, $E\oplus\mathbb{R}$ is a vector subspace of $C\left(  X_{\infty
}\right)  $ containing the constant functions on $X_{\infty}$ and separating
the points of $X_{\infty}$. Moreover, an easy calculation reveals that the
positive part $\left(  f+r\right)  ^{+}$ in $C\left(  X_{\infty}\right)  $ of
a function $f+r\in E\oplus\mathbb{R}$ is given by%
\[
\left(  f+r\right)  ^{+}=\left\{
\begin{array}
[c]{l}%
f^{+}-r\left(  1\wedge\dfrac{1}{r}f^{-}\right)  +r\text{ if }r>0\\
\\
f^{+}+r\left(  1\wedge\dfrac{-1}{r}f^{+}\right)  \text{ if }r<0\\
\\
f^{+}\text{ if }r=0\text{.}%
\end{array}
\right.
\]
Since $E$ is a truncated Riesz subspace of $C_{0}\left(  X\right)  $, the
direct sum $E\oplus\mathbb{R}$ is a Riesz subspace of $C\left(  X_{\infty
}\right)  $. Using the classical Stone-Weierstrass for compact Hausdorff
spaces (see, for instance, Theorem 2.1.1 in \cite{M91}), we derive that
$E\oplus\mathbb{R}$ is uniformly dense in $C\left(  X_{\infty}\right)  $.
Accordingly, if $f\in C_{0}\left(  X\right)  $ then there exist sequences
$\left(  f\right)  _{n}$ in $E$ and $\left(  r_{n}\right)  $ in $\mathbb{R}$
such that $\lim\left(  f_{n}+r_{n}\right)  =f$. Hence, for $n\in\left\{
1,2,...\right\}  $, we have%
\[
\left\vert r_{n}\right\vert =\left\vert f_{n}\left(  \infty\right)
+r_{n}-f\left(  \infty\right)  \right\vert \leq\left\Vert f_{n}+r_{n}%
-f\right\Vert _{\infty}.
\]
Thus, $\lim r_{n}=0$ and so $\lim f_{n}=f$. This means that $E$ is uniformly
dense in $C_{0}\left(  X\right)  $.

We now focus on the `\textit{only if}'. Suppose that $E$ is uniformly dense in
$C_{0}\left(  X\right)  $. Obviously, $C_{0}\left(  X\right)  $ vanishes
nowhere and so does $E$. We claim that $E$ separates the points of
$C_{0}\left(  X\right)  $. To this end, pick $x,y\in X$ with $x\neq y$. Since
$C_{0}\left(  X\right)  $ separates the points of $X$, there exists $f\in
C_{0}\left(  X\right)  $ such that $f\left(  x\right)  \neq f\left(  y\right)
$. Using the density condition, there exists $g\in E$ such that $\left\Vert
f-g\right\Vert _{\infty}<\left\vert f\left(  x\right)  -f\left(  y\right)
\right\vert /2$. Consequently,%
\begin{align*}
\left\vert g\left(  x\right)  -g\left(  y\right)  \right\vert  &  =\left\vert
g\left(  x\right)  -f\left(  x\right)  +f\left(  x\right)  -f\left(  y\right)
+f\left(  y\right)  -g\left(  y\right)  \right\vert \\
&  \geq\left\vert f\left(  x\right)  -f\left(  y\right)  \right\vert
-2\left\Vert f-g\right\Vert _{\infty}>0.
\end{align*}
We get $g\left(  x\right)  \neq g\left(  y\right)  $, which completes the
proof of the lemma.
\end{proof}

We have gathered at this point all the ingredients for the main result of this paper.

\begin{theorem}
\label{main}Let $E$ be a strongly truncated Riesz space with no nonzero
$^{\ast}$-infinitely small elements. Then the map $T:E\rightarrow C_{0}\left(
\eta E\right)  $ defined by%
\[
T\left(  f\right)  \left(  u\right)  =u\left(  f\right)  \text{, for all }f\in
E\text{ and }u\in\eta E
\]
is an injective $^{\ast}$-homomorphism with uniformly dense range.
\end{theorem}

\begin{proof}
Pick $f\in E$ and define the evaluation $\delta_{f}:\eta E\rightarrow
\mathbb{R}$ by putting%
\[
\delta_{f}\left(  u\right)  =u\left(  f\right)  \text{, for all }u\in\eta E.
\]
We claim that $\delta_{f}\in C_{0}\left(  \eta E\right)  $. To this end,
observe that $\delta_{f}$ is continuous since it is the restriction of the
projection $\pi_{f}$ to $\eta E$. Moreover, if $\varepsilon\in\left(
0,\infty\right)  $ then%
\[
K\left(  \delta_{f},\varepsilon\right)  =\left\{  u\in\eta E:\left\vert
\delta_{f}\left(  u\right)  \right\vert \geq\varepsilon\right\}  \subset\eta
E\cup\left\{  0\right\}  .
\]
As $\eta E\cup\left\{  0\right\}  $ is compact (see the proof of Lemma
\ref{loc}), so is $K\left(  \delta_{f},\varepsilon\right)  $ (notice that
$K\left(  \delta_{f},\varepsilon\right)  $ is a closed set in $\eta
E\cup\left\{  0\right\}  $). It follows that $\delta_{f}$ vanishes at
infinity, as desired. Accordingly, the map $T:E\rightarrow C_{0}\left(  \eta
E\right)  $ given by%
\[
T\left(  f\right)  =\delta_{f}\text{, for all }f\in E
\]
is well-defined. It is an easy task to show that $T$ is a Riesz homomorphism.
Moreover, since $E$ has no nonzero $^{\ast}$-infinitely small elements,
Theorem \ref{spectrum} yields directly that $T$ is one-to-one. Furthermore, if
$f\in E^{+}$ and $u\in\eta E$, then%
\begin{align*}
\left(  1\wedge T\left(  f\right)  \right)  \left(  u\right)   &  =\left(
1\wedge\delta_{f}\right)  \left(  u\right)  =\min\left\{  1,\delta_{f}\left(
u\right)  \right\} \\
&  =\min\left\{  1,u\left(  f\right)  \right\}  =u\left(  f^{\ast}\right)
=T\left(  f^{\ast}\right)  \left(  u\right)  .
\end{align*}
Since $u$ is arbitrary in $\eta E$, we get%
\[
T\left(  f^{\ast}\right)  =1\wedge T\left(  f\right)  \text{, for all }f\in
E^{+}.
\]
This means that $T$ is $^{\ast}$-homomorphism. It remains to shows that the
range $\operatorname{Im}\left(  T\right)  $ of $T$ is uniformly dense in
$C_{0}\left(  \eta E\right)  $. For, let $u,v\in\eta E$ such that $u\neq v$.
Hence, $u\left(  f\right)  \neq v\left(  f\right)  $ for some $f\in E$.
Therefore,%
\[
\delta_{f}\left(  u\right)  =u\left(  f\right)  \neq v\left(  f\right)
=\delta_{f}\left(  v\right)
\]
from which it follows that $\operatorname{Im}\left(  T\right)  $ separates the
points of $\eta E$. Moreover, since $\eta E$ does not contain the zero
homomorphism, $\operatorname{Im}\left(  T\right)  $ vanishes nowhere. This
together with Lemma \ref{loc} and Lemma \ref{SW} completes the proof.
\end{proof}

The following remark deserves to be empathized.

\begin{remark}
\label{Rq}\emph{Under the conditions of Theorem \ref{main}, we may consider
}$E$\emph{ as a normed subspace of }$C_{0}\left(  \eta E\right)  $\emph{. In
this situation, it is readily checked that the closed unit ball }$B_{\infty}%
$\emph{ of }$E$\emph{ coincides with the set of all }$f\in E$\emph{ such that
}$\left\vert f\right\vert \in P_{\ast}\left(  E\right)  $\emph{, i.e.,}%
\[
B_{\infty}=\left\{  f\in E:\left\Vert f\right\Vert _{\infty}\leq1\right\}
=\left\{  f\in E:\left\vert f\right\vert \in P_{\ast}\left(  E\right)
\right\}  .
\]

\end{remark}

We end this section with the following observation. From Theorem \ref{main} it
follows directly that any strongly truncated Riesz space $E$ with no nonzero
$^{\ast}$-infinitely small elements is Archimedean. It is plausible therefore
to think that any truncated Riesz space with no nonzero $^{\ast}$-infinitely
small elements is Archimedean. However, the next example shows that this is
not true.

\begin{example}
Assume that the Euclidean plan $E=\mathbb{R}^{2}$ is furnished with its
lexicographic ordering. We know that $E$ is a non-Archimedean Riesz space.
Clearly, the formula%
\[
\left(  x,y\right)  ^{\ast}=\left(  x,y\right)  \wedge\left(  0,1\right)
,\text{ for all }\left(  x,y\right)  \in E^{2}%
\]
defines a truncation on $E$. Let $\left(  x,y\right)  \in E^{+}$ such that
$\varepsilon\left(  x,y\right)  \in P_{\ast}\left(  E\right)  $ for all
$\varepsilon\in\left(  0,\infty\right)  $. We have%
\[
\left(  \varepsilon x,\varepsilon y\right)  =\varepsilon\left(  x,y\right)
=\left[  \varepsilon\left(  x,y\right)  \right]  ^{\ast}=\left(  \varepsilon
x,\varepsilon y\right)  \wedge\left(  0,1\right)  ,
\]
which means that $\left(  \varepsilon x,\varepsilon y\right)  \leq\left(
0,1\right)  $ for all $\varepsilon\in\left(  0,\infty\right)  $. Assume that
$\varepsilon x<0$ for some $\varepsilon\in\left(  0,\infty\right)  $. Then
$x<0$ which is impossible since $\left(  x,y\right)  \geq\left(  0,0\right)
$. Thus, $\varepsilon x=0$ and $\varepsilon y\leq1$ for all $\varepsilon
\in\left(  0,\infty\right)  $. Therefore, $x=0$ and $y\leq0$. But then $x=y=0$
because $\left(  x,y\right)  \in E^{+}$. Accordingly, $E$ is a non-Archimedean
truncated Riesz space with no non-trivial $^{\ast}$-infinitely small elements.
\end{example}

Notice finally that the truncation in the above example is not strong.

\section{Uniform completeness with respect to a truncation}

In this section, our purpose is to find a necessary and sufficient condition
on the strongly truncated Riesz space $E$ with no nonzero $^{\ast}$-infinitely
small elements for $T$ in Theorem \ref{main} to be a $^{\ast}$-isomorphism. As
it could be expected, what we need is a certain condition of completeness. We
proceed to the details.

Let $E$ be a truncated Riesz space. A sequence $\left(  f_{n}\right)  $ in $E$
is said to $^{\ast}$-\textsl{converge }(or, to be $^{\ast}$%
-\textsl{convergent}) in $E$ if there exists $f\in E$ such that, for every
$\varepsilon\in\left(  0,\infty\right)  $ there is $n_{\varepsilon}\in\left\{
1,2,...\right\}  $ for which
\[
\varepsilon\left\vert f_{n}-f\right\vert \in P_{\ast}\left(  E\right)  \text{,
for all }n\in\left\{  n_{\varepsilon},n_{\varepsilon}+1,...\right\}  .
\]
Such an element $f$ is called a $^{\ast}$-\textsl{limit} of the sequence\emph{
}$\left(  f_{n}\right)  $ in $E$. As we shall see next, $^{\ast}$-limits are
unique, provided $E\ $has no nonzero $^{\ast}$-infinitely small elements.

\begin{proposition}
Let $E$ be a truncated Riesz space. Then any sequence in $E$ has at most one
$^{\ast}$-limit if and only if $E$ has no nonzero $^{\ast}$-infinitely small elements.
\end{proposition}

\begin{proof}
\textsl{Sufficiency. }Choose a sequence $\left(  f_{n}\right)  $ in $E$ with
two $^{\ast}$-limits $f$ and $g$ in $E$. Let $\varepsilon\in\left(
0,\infty\right)  $ and $n_{1},n_{2}\in\left\{  1,2,...\right\}  $ such that%
\[
2\varepsilon\left\vert f_{n}-f\right\vert \in P_{\ast}\left(  E\right)
\text{, for all }n\in\left\{  n_{1},n_{1}+1,...\right\}
\]
and%
\[
2\varepsilon\left\vert f_{n}-g\right\vert \in P_{\ast}\left(  E\right)
\text{, for all }n\in\left\{  n_{2},n_{2}+1,...\right\}  .
\]
Put $n_{0}=\max\left\{  n_{1},n_{2}\right\}  $ and observe that $2\varepsilon
\left\vert f_{n_{0}}-f\right\vert \in P_{\ast}\left(  E\right)  $ and
$2\varepsilon\left\vert f_{n_{0}}-g\right\vert \in P_{\ast}\left(  E\right)
$. This together with Lemma \ref{elem} $\mathrm{(v)}$ yields that%
\[
\varepsilon\left\vert f-g\right\vert \leq\varepsilon\left(  \left\vert
f_{n_{0}}-f\right\vert +\left\vert f_{n_{0}}-g\right\vert \right)
\leq2\varepsilon\left\vert f_{n_{0}}-f\right\vert \vee2\varepsilon\left\vert
f_{n_{0}}-g\right\vert \in P_{\ast}\left(  E\right)  .
\]
Hence, $\varepsilon\left\vert f-g\right\vert \in P_{\ast}\left(  E\right)  $
(where we use Lemma \ref{elem} $\mathrm{(iii)}$) . As $\varepsilon$ is
arbitrary in $\left(  0,\infty\right)  $, we get $f=g$ and the sufficiency follows.

\textsl{Necessity.} Pick a $^{\ast}$-infinitely small element $f\in E$ and put%
\[
f_{n}=\frac{1}{n}f\text{, for all }n\in\left\{  1,2,...\right\}  .
\]
If $\varepsilon\in\left(  0,\infty\right)  $ and $n\in\left\{
1,2,...\right\}  $, then%
\[
\varepsilon\left\vert f_{n}\right\vert =\frac{\varepsilon}{n}\left\vert
f\right\vert \in P_{\ast}\left(  E\right)  .
\]
This yields that $0$ is a $^{\ast}$-limit of $\left(  f_{n}\right)  $ in $E$.
Analogously, for $\varepsilon\in\left(  0,\infty\right)  $ and $n\in\left\{
1,2,...\right\}  $, we have%
\[
\varepsilon\left\vert f_{n}-f\right\vert =\varepsilon\left(  1-\frac{1}%
{n}\right)  \left\vert f\right\vert \in P_{\ast}\left(  E\right)  .
\]
This shows that $f$ is a $^{\ast}$-limit of $\left(  x_{n}\right)  $ in $E$.
By uniqueness of $^{\ast}$-limits, we conclude $f=0$ and the proposition follows.
\end{proof}

A sequence $\left(  f_{n}\right)  $ in the truncated Riesz space $E$ is called
a $^{\ast}$-\textsl{Cauchy sequence} if, for every $\varepsilon\in\left(
0,\infty\right)  $, there exists $n_{\varepsilon}\in\left\{  1,2,...\right\}
$\emph{ }such that%
\[
\varepsilon\left\vert f_{m}-f_{n}\right\vert \in P_{\ast}\left(  E\right)
\text{, for all }m,n\in\left\{  n_{\varepsilon},n_{\varepsilon}+1,...\right\}
.
\]
We then say that $E$ is said to be $^{\ast}$-\textsl{uniformly complete} if
any $^{\ast}$-Cauchy sequence in $E$ is $^{\ast}$-convergent in $E$. Let's
give a simple example.

\begin{example}
\label{exp}Let $X$ be a locally compact Hausdorff space and assume that the
Banach lattice $E=C_{0}\left(  X\right)  $ is equipped with its strong
truncation given by%
\[
f^{\ast}=1\wedge f\text{, for all }f\in E.
\]
Clearly, $E$ has no nonzero $^{\ast}$-infinitely small elements. Let $\left(
f_{n}\right)  $ be a $^{\ast}$-Cauchy sequence in $E$. For $\varepsilon
\in\left(  0,\infty\right)  $, we can find $n_{\varepsilon}\in\left\{
1,2,...\right\}  $ such that%
\[
\varepsilon\left\vert f_{m}-f_{n}\right\vert \in P_{\ast}\left(  E\right)
\text{, for all }m,n\in\left\{  n_{\varepsilon},n_{\varepsilon}+1,...\right\}
.
\]
Whence, if $m,n\in\left\{  n_{\varepsilon},n_{\varepsilon}+1,...\right\}  $
then%
\[
\varepsilon\left\vert f_{m}\left(  x\right)  -f_{n}\left(  x\right)
\right\vert \leq1\text{, for all }x\in X.
\]
But then $\varepsilon\left\Vert f_{m}-f_{n}\right\Vert _{\infty}\leq1$ and so
$\left(  f_{n}\right)  $ is a norm Cauchy sequence in $E$. As $E$ is norm
complete, we derive that $\left(  f_{n}\right)  $ has a norm limit $f$ in $E$.
Therefore, if $n\in\left\{  n_{\varepsilon},n_{\varepsilon}+1,...\right\}  $
and $x\in X$ then%
\[
\varepsilon\left\vert f_{n}\left(  x\right)  -f\left(  x\right)  \right\vert
\leq\varepsilon\left\Vert f_{n}-f\right\Vert _{\infty}\leq1.
\]
This yields directly that%
\[
\varepsilon\left\vert f_{n}-f\right\vert \in P_{\ast}\left(  E\right)  \text{,
for all }n\in\left\{  n_{\varepsilon},n_{\varepsilon}+1,...\right\}  .
\]
It follows that $\left(  f_{n}\right)  $ is $^{\ast}$-convergent to $f$ in
$E$, proving that $E\ $is $^{\ast}$-uniformly complete.
\end{example}

The following is the main result of this section.

\begin{theorem}
\label{complete}Let $E$ be a $^{\ast}$-uniformly complete strongly truncated
Riesz space with no nonzero $^{\ast}$-infinitely small elements. Then the map
$T:E\rightarrow C_{0}\left(  \eta E\right)  $ defined by%
\[
T\left(  f\right)  \left(  u\right)  =u\left(  f\right)  \text{, for all }f\in
E\text{ and }u\in\eta E
\]
is a $^{\ast}$-isomorphism.
\end{theorem}

\begin{proof}
In view Theorem \ref{main}, the proof would be complete once we show that
$T\left(  E\right)  $ is uniformly closed in $C_{0}\left(  \eta E\right)  $.
Hence, let $\left(  f_{n}\right)  $ be a sequence in $E$ such that $\left(
T\left(  f_{n}\right)  \right)  $ converges uniformly to $g\in C_{0}\left(
\eta E\right)  $. It follows that $\left(  T\left(  f_{n}\right)  \right)  $
is a uniformly Cauchy sequence in $C_{0}\left(  \eta E\right)  $. So, if
$\varepsilon\in\left(  0,\infty\right)  $ then there exists $n_{\varepsilon
}\in\left\{  1,2,...\right\}  $ such that, whenever $m,n\in\left\{
n_{\varepsilon},n_{\varepsilon}+1,...\right\}  $, we have $\left\Vert T\left(
f_{m}\right)  -T\left(  f_{n}\right)  \right\Vert _{\infty}<1/\varepsilon$.
Choose $u\in\eta E$ and $m,n\in\left\{  n_{\varepsilon},n_{\varepsilon
}+1,...\right\}  $. We derive that%
\[
\left\vert T\left(  f_{m}\right)  \left(  u\right)  -T\left(  f_{n}\right)
\left(  u\right)  \right\vert \leq\left\Vert T\left(  f_{m}\right)  -T\left(
f_{n}\right)  \right\Vert _{\infty}<1/\varepsilon.
\]
Therefore,%
\[
u\left(  \left\vert f_{m}-f_{n}\right\vert \right)  =\left\vert T\left(
f_{m}\right)  \left(  u\right)  -T\left(  f_{n}\right)  \left(  u\right)
\right\vert <1/\varepsilon.
\]
So,%
\[
u\left(  \left(  \varepsilon\left\vert f_{m}-f_{n}\right\vert \right)  ^{\ast
}\right)  =\min\left\{  1,u\left(  \varepsilon\left\vert f_{m}-f_{n}%
\right\vert \right)  \right\}  =u\left(  \varepsilon\left\vert f_{m}%
-f_{n}\right\vert \right)  .
\]
Since $u$ is arbitrary in $\eta E$, Lemma \ref{spectrum} yields that
$\varepsilon\left\vert f_{m}-f_{n}\right\vert \in P_{\ast}\left(  E\right)  $,
meaning that $\left(  f_{n}\right)  $ is a $^{\ast}$-Cauchy sequence in $E$.
This together with the $^{\ast}$-uniform completeness of $E$ yields that
$\left(  f_{n}\right)  $ is $^{\ast}$-convergent to some $f\in E$. Hence, we
get%
\[
\varepsilon\left\vert f_{n}-f\right\vert \in P_{\ast}\left(  E\right)  \text{,
for all }n\in\left\{  n_{\varepsilon},n_{\varepsilon}+1,...\right\}  .
\]
So, if $n\in\left\{  n_{\varepsilon},n_{\varepsilon}+1,...\right\}  $ and
$u\in\eta E$ then%
\begin{align*}
\left\vert T\left(  f_{n}\right)  \left(  u\right)  -T\left(  f\right)
\left(  u\right)  \right\vert  &  =\frac{1}{\varepsilon}u\left(
\varepsilon\left\vert f_{n}-f\right\vert \right) \\
&  =\frac{1}{\varepsilon}u\left(  \left(  \varepsilon\left\vert f_{n}%
-f\right\vert \right)  ^{\ast}\right)  \leq\frac{1}{\varepsilon}.
\end{align*}
We quickly get $\left\Vert T\left(  f_{n}\right)  -T\left(  f\right)
\right\Vert _{\infty}\leq1/\varepsilon$ from which it follows that $\left(
T\left(  f_{n}\right)  \right)  $ converges uniformly to $T\left(  f\right)
$. By uniqueness, we get $g=T\left(  f\right)  \in T\left(  E\right)  $.
Consequently, $T\left(  E\right)  $ is closed in $C_{0}\left(  \eta E\right)
$, which is the desired result.
\end{proof}

As for the remark at the end of the previous section, it follows quite easily
from Theorem \ref{complete} that any $^{\ast}$-uniformly complete strongly
truncated Riesz space with no nonzero elements is relatively uniformly
complete (in the usual sense \cite[Pages 19,20]{L79} or \cite[Page 248]%
{LZ71}). Indeed, it is well known that any Banach lattice (and so
$C_{0}\left(  X\right)  $) is relatively uniformly complete (see, for
instance, Theorem 15.3 in \cite{Z97}). Nevertheless, a $^{\ast}$-uniformly
complete truncated Riesz space with no nonzero $^{\ast}$-infinitely small
elements need not be relatively uniformly complete. An example in this
direction is provided next.

\begin{example}
Let $X=\left\{  0\right\}  \cup\left\{  1/2\right\}  \cup\left[  1,2\right]  $
with its usual topology. It is routine to show that the set $E$ of all
piecewise polynomial functions $f$ in $C\left(  X\right)  $ with $f\left(
1/2\right)  =0$ is a Riesz subspace of $C\left(  X\right)  $. Clearly, the
formula%
\[
f^{\ast}=f\wedge\mathcal{X}_{\left\{  0,\frac{1}{2}\right\}  }\text{, for all
}f\in E
\]
defines a \emph{(}non strong\emph{)} truncation on $E$, where $\mathcal{X}%
_{\left\{  0,\frac{1}{2}\right\}  }$ is the characteristic function of the
pair $\left\{  0,1/2\right\}  $. A short moment's though reveals that $E$ has
no non trivial $^{\ast}$-infinitely small elements. Consider now a $^{\ast}%
$-Cauchy sequence $\left(  f_{n}\right)  $ in $E$. Given $\varepsilon
\in\left(  0,\infty\right)  $, there is $n_{0}\in\left\{  1,2,...\right\}  $
for which%
\[
\left\vert f_{n}-f_{n_{0}}\right\vert \leq\varepsilon\mathcal{X}_{\left\{
0,\frac{1}{2}\right\}  }\text{, for all }n\in\left\{  n_{0},n_{0}%
+1,...\right\}  .
\]
Thus, if $n\in\left\{  n_{0},n_{0}+1,...\right\}  $ and $x\in\left\{
1/2\right\}  \cup\left[  1,2\right]  $ then%
\[
f_{n}\left(  x\right)  =f_{n_{0}}\left(  x\right)  \text{\quad and\quad
}\left\vert f_{n}\left(  0\right)  -f_{n_{0}}\left(  0\right)  \right\vert
\leq\varepsilon.
\]
In particular, $\left(  f_{n}\left(  0\right)  \right)  $ is Cauchy sequence
in $\mathbb{R}$ from which it follows that $\left(  f_{n}\left(  0\right)
\right)  $ converges to some real number $a$. This yields quickly that
$\left(  f_{n}\right)  $ is $^{\ast}$-convergent in $E$ to a function $f\in E$
given by%
\[
f\left(  0\right)  =a\text{\quad and\quad}f\left(  x\right)  =f_{n_{0}}\left(
x\right)  \text{ for all }x\in\left\{  1/2\right\}  \cup\left[  1,2\right]  .
\]
We conclude that $E$ is $^{\ast}$-uniformly complete. At this point, we show
that $E$ is not relatively uniformly complete. Indeed, by the Weierstrass
Approximation Theorem, there exists a polynomial sequence $\left(
p_{n}\right)  $ which converges uniformly on $\left[  1,2\right]  $ to the
function $f$ defined by
\[
f\left(  x\right)  =\sqrt{x}\text{, for all }x\in\left[  0,\infty\right)  .
\]
Define a sequence $\left(  q_{n}\right)  $ in $E$ by%
\[
q_{n}=p_{n}\chi_{\left[  1,2\right]  }\text{, for all }n\in\left\{
1,2,...\right\}  .
\]
The uniform limit of $\left(  q_{n}\right)  $ in $C\left(  X\right)  $ is the
function $f\chi_{\left[  1,2\right]  }$. Assume that $\left(  q_{n}\right)  $
converges relatively uniformly in $E$ to a function $g\in E$. So, there exists
$h\in E$ such that, for every $\varepsilon\in\left(  0,\infty\right)  $ there
is $n_{\varepsilon}\in\left\{  1,2,...\right\}  $ for which%
\[
\left\vert q_{n}-g\right\vert \leq\varepsilon h\text{, for all }n\in\left\{
n_{\varepsilon},n_{\varepsilon}+1,...\right\}  .
\]
This leads straightforwardly to the contradiction $f\chi_{\left[  1,2\right]
}=g\in E$, meaning that $E$ is not relatively uniformly complete.
\end{example}

\section{Applications}

The main purpose of this section is to show how we can apply our central
result (Theorem \ref{main}) to derive representation theorems from existing
literature. We will be interested first in the classical Kakutani
Representation Theorem (see, for instance, Theorem 45.3 in \cite{LZ71}),
namely, for any Archimedean Riesz space with a strong unit $e$, there exists a
compact Hausdorff space $K$ such that $E$ and $C\left(  K\right)  $ are Riesz
isomorphic so that $e$ is identified with the constant function $1$ on $K$.
Recall here that a positive element $e$ in a Riesz space $E$ is called a
\textsl{strong unit} if, for every $f\in E$ there exists $\varepsilon
\in\left(  0,\infty\right)  $ such that $\left\vert f\right\vert
\leq\varepsilon e$. The Kakutani Representation Theorem ensues from our main
theorem as we shall see right now.

\begin{corollary}
Let $E$ be an Archimedean Riesz space with a strong unit $e>0$. Then $\eta E$
is compact and $E$ is Riesz isomorphic to a uniformly dense Riesz subspace of
$C\left(  \eta E\right)  $ in such a way that $e$ is identified with the
constant function $1$ on $\eta E$.
\end{corollary}

\begin{proof}
It is readily checked that $E$ is a strongly truncated Riesz space under the
truncation given by%
\[
f^{\ast}=e\wedge f\text{, for all }f\in E^{+}.
\]
Moreover, since $E$ is Archimedean and $e$ is a strong unit, we can easily
verify that $E$ has no nonzero $^{\ast}$-infinitely small elements. Now, let
$u\in\mathbb{R}^{E}$. We claim that $u\in\eta E$ if and only if $u$ is a Riesz
homomorphism with $u\left(  e\right)  =1$. To this end, assume that $u$ is
Riesz homomorphism with $u\left(  e\right)  =1$. So, if $f\in E^{+}$ then%
\[
u\left(  f^{\ast}\right)  =u\left(  e\wedge f\right)  =\min\left\{  u\left(
e\right)  ,u\left(  f\right)  \right\}  =\min\left\{  1,u\left(  f\right)
\right\}  .
\]
We derive that $u\in\eta E$. Conversely, suppose that $u\in\eta E$. We have to
show that $u\left(  e\right)  =1$. Observe that%
\[
u\left(  e\right)  =u\left(  e\wedge e\right)  =u\left(  e^{\ast}\right)
=\min\left\{  1,u\left(  e\right)  \right\}  \leq1.
\]
Moreover, we have $u\neq0$ and so $u\left(  f\right)  \neq0$ for some $f\in
E^{+}$. By dividing by $u\left(  f\right)  $ if necessary, we can assume that
$u\left(  f\right)  =1$. Thus,%
\begin{align*}
1  &  =\min\left\{  1,u\left(  f\right)  \right\}  =u\left(  f^{\ast}\right)
=u\left(  e\wedge f\right) \\
&  =\min\left\{  u\left(  e\right)  ,u\left(  f\right)  \right\}
=\min\left\{  u\left(  e\right)  ,1\right\}  .
\end{align*}
Therefore, $u\left(  e\right)  \geq1$ and thus $u\left(  e\right)  =1$. It
follows therefore that $\eta E$ is a closet set in the compact space $\eta
E\cup\left\{  0\right\}  $ (see the proof of Lemma \ref{loc}) and so $\eta E$
is compact. We derive in particular that $C\left(  \eta E\right)
=C_{0}\left(  \eta E\right)  $. This together with Theorem \ref{main} yields
that $E$ is Riesz isomorphic to a dense truncated Riesz subspace of $C\left(
\eta E\right)  $ \textit{via} the map $T:E\rightarrow C\left(  \eta E\right)
$ defined by%
\[
T\left(  f\right)  \left(  u\right)  =u\left(  f\right)  \text{, for all }%
u\in\eta E\text{ and }f\in E.
\]
Hence, if $u\in\eta E$ then%
\[
T\left(  e\right)  \left(  u\right)  =u\left(  e\right)  =1.
\]
This completes the proof of the corollary.
\end{proof}

In the next paragraph, we discuss another representation theorem obtained by
Fremlin in \cite[83L (d)]{F74}. A norm $\left\Vert .\right\Vert $ on a Riesz
space $E$ is called a \textsl{Riesz} (or, \textsl{lattice}) \textsl{norm}
whenever $\left\vert f\right\vert \leq\left\vert g\right\vert $ in $E$ implies
$\left\Vert f\right\Vert \leq\left\Vert g\right\Vert $. The Riesz norm on $E$
is called a \textsl{Fatou norm} if for every increasing net $\left(
f_{a}\right)  _{a\in A}$ in $E$ with supremum $f\in E$ it follows that%
\[
\left\Vert f\right\Vert =\sup\left\{  \left\Vert f_{a}\right\Vert :a\in
A.\right\}
\]
Furthermore, the Riesz norm on $E$ is called an $M$-\textsl{norm} if%
\[
\left\Vert f\vee g\right\Vert =\max\left\{  \left\Vert f\right\Vert
;\left\Vert g\right\Vert \right\}  \text{, for all }f,g\in E^{+}.
\]
Fremlin proved that if $E$ is a Riesz space with a Fatou $M$-norm such that
the supremum%
\[
\sup\left\{  g\in E:0\leq g\leq f\text{ and }\left\Vert g\right\Vert
\leq\alpha\right\}
\]
exists in $E$ for every $f\in E^{+}$ and $\alpha\in\left(  0,\infty\right)  $,
then $E$ is isomorphic, as a normed Riesz space, to a truncated Riesz subspace
of $\ell^{\infty}\left(  X\right)  $ for some nonvoid set $X$. Here,
$\ell^{\infty}\left(  X\right)  $ denotes the Riesz space of all bounded
real-valued functions on $X$. As we shall see in our last result, our main
theorem allows as to make the conclusion by Fremlin more precise by showing
that, actually, $E$ is uniformly dense in a $C_{0}\left(  X\right)  $-type space.

\begin{corollary}
Let $E$ be a Riesz space with a Fatou $M$-norm such that the supremum%
\[
\sup\left\{  g\in E:0\leq g\leq f\text{ and }\left\Vert g\right\Vert
\leq1\right\}
\]
exists in $E$ for every $f\in E^{+}$. Then $E$ is isomorphic, as a normed
Riesz space, to a uniformly dense truncated Riesz subspace of $C_{0}\left(
\eta E\right)  $.
\end{corollary}

\begin{proof}
Put%
\[
f^{\ast}=\sup\left\{  g\in E:0\leq g\leq f\text{ and }\left\Vert g\right\Vert
\leq1\right\}  \text{, for all }f\in E^{+}.
\]
It turns out that this equality gives rise to a truncation on $E$. Indeed, let
$f,g\in E^{+}$ and put%
\[
U=\left\{  f\wedge h:0\leq h\leq g\text{ and }\left\Vert h\right\Vert
\leq1\right\}
\]
and%
\[
V=\left\{  g\wedge h:0\leq h\leq f\text{ and }\left\Vert h\right\Vert
\leq1\right\}  .
\]
If $k$ is an upper bound of $U$ and $h\in\left[  0,f\right]  $ with
$\left\Vert h\right\Vert \leq1$, then%
\[
0\leq h\wedge g\leq g\text{\quad and\quad}\left\Vert h\wedge g\right\Vert
\leq\left\Vert h\right\Vert \leq1.
\]
Hence,%
\[
k\geq\left(  h\wedge g\right)  \wedge f=\left(  h\wedge f\right)  \wedge
g=h\wedge g.
\]
This means that $k$ is an upper bound of $V$. We derive quickly that $U$ and
$V$ have the same upper bounds, and so the same supremum (which exists in
$E$). Therefore, $f^{\ast}\wedge g=f\wedge g^{\ast}$, meaning that $\ast$ is a
truncation on $E$. We claim that $E$ has no nonzero $^{\ast}$-infinitely small
elements. To this end, take $f\in E^{+}$ with $\left(  \varepsilon f\right)
^{\ast}=\varepsilon f$ for all $\varepsilon\in\left(  0,\infty\right)  $.
Since $E$ has a Fatou $M$-norm, we can write, for $\varepsilon\in\left(
0,\infty\right)  $,%
\[
\varepsilon\left\Vert f\right\Vert =\left\Vert \left(  \varepsilon f\right)
^{\ast}\right\Vert =\sup\left\{  \left\Vert g\right\Vert :0\leq g\leq
\varepsilon f\text{ and }\left\Vert g\right\Vert \leq1\right\}  \leq1,
\]
so $f=0$, as desired. Now, we prove that the truncation $\ast$ is strong. If
$f\in E^{+}$ with $f\neq0$ then%
\begin{align*}
\left(  \frac{f}{\left\Vert f\right\Vert }\right)  ^{\ast}  &  =\sup\left\{
h:0\leq h\leq\frac{f}{\left\Vert f\right\Vert }\text{ and }\left\Vert
h\right\Vert \leq1\right\} \\
&  =\sup\left\{  h:0\leq h\leq\frac{f}{\left\Vert f\right\Vert }\right\}
=\frac{f}{\left\Vert f\right\Vert }.
\end{align*}
This gives the required fact. Consequently, using Theorem \ref{main}, we infer
that $E$ is isomorphic as a truncated Riesz space to a uniformly dense
truncated Riesz subspace of $C_{0}\left(  \eta E\right)  $. In the next lines,
we shall identify $E$ with its isomorphic copy in $C_{0}\left(  \eta E\right)
$. So, it remains to prove that the norm of $E$ coincides with the uniform
norm. To see this, it suffices to show that%
\[
B_{E}=\left\{  f\in E:\left\Vert f\right\Vert \leq1\right\}  =B_{\infty
}=\left\{  f\in E:\left\Vert f\right\Vert _{\infty}\leq1\right\}  .
\]
Pick $f\in B_{E}\ $and observe that%
\[
f=\sup\left\{  g\in E:0\leq g\leq f\text{ and }g\in B_{E}\right\}  =f^{\ast}.
\]
However,%
\[
\left\Vert f^{\ast}\right\Vert _{\infty}=\left\Vert 1\wedge f\right\Vert
_{\infty}\leq1.
\]
It follows that $f\in B_{\infty}$ from which we derive that $B_{\infty}$
contains $B_{E}$. Conversely, if $0\leq f\in B_{\infty}$ then, by Remark
\ref{Rq}, $f\in P_{\ast}\left(  E\right)  $. But then%
\[
f=\sup\left\{  g\in E:0\leq g\leq f\text{ and }g\in B_{E}\right\}  .
\]
This yields easily that $\left\Vert f\right\Vert \leq1$, completing the proof
of the corollary.
\end{proof}

\medskip

\noindent\textbf{Acknowledgment. }This research is supported by Research
Laboratory LATAO Grant LR11ES12.


\begin{thebibliography}{99}                                                                                               %


\bibitem {AB06}C. D. Aliprantis and O. Burkinshaw, \textit{Positive
Operators}, Springer-Verlag, Dordrecht, 2006.

\bibitem {AC07}F. Altomare and M. Cappelletti Montano, On some density
theorems in regular vector lattices of continuous functions, \textit{Collect.
Math}. 58 (2007), 131--149.

\bibitem {B14}R. N. Ball, Truncated abelian lattice-ordered groups I: The
pointed (Yosida) representation, \textit{Topology Appl}., 162 (2014), 43--65.

\bibitem {B14-0}R. N. Ball, Truncated abelian lattice-ordered groups II: The
the pointfree (Madden) representation, \textit{Topology. Appl.}, 162 (2014), 56-86.

\bibitem {BE17}K. Boulabiar and C. El Adeb, Unitization of Ball truncated
$\ell$-groups, \textit{Algebra Univ}., 78 (2017) 93--104.

\bibitem {BHM18}K. Boulabiar, H. Hafsi, and M. Mahfoudhi, Alexandroff
unitization of a truncated vector lattice, \textit{Algebra Univ.}, (2018) 79:48.

\bibitem {D04}R. M. Dudley, \textit{Real Analysis and Probability}, Cambridge
Univ. Press, Cambridge-New York, 2004.

\bibitem {E89}R. Engelking, \textit{General Topology}, Heldermann-Verlag,
Berlin, 1989.

\bibitem {F74}D. H. Fremlin, \textit{Topological Riesz spaces and Measure
Theory}, Cambridge Univ. Press, Cambridge, 1974.

\bibitem {F06}D. H. Fremlin, \textit{Measure Theory}. Vol. 4, Torres Fremlin,
Colchester, 2006.

\bibitem {HP84}C. B. Huijsmans and B. de Pagter, Subalgebras and Riesz
subspaces of an $f$-algebras, \textit{Proc. London Math. Soc.}, 48 (1984), 161-174.

\bibitem {JF87}B. R. F. Jefferies and D. H. Fremlin. An indecomposable Daniell
integral. \textit{Proc. Am. Math. Soc}., 101 (1987), 101-647--651.

\bibitem {K97}H. K\"{o}nig, \textit{Measure and Integration}, Springer-Verlag,
Berlin-Heidelberg-New York, 1997.

\bibitem {L79}W. A. J. Luxemburg, \textit{Some aspects of the theory of Riesz
spaces}, Univ. Arkansas, Lecture Notes Math., 4. University of Arkansas,
Fayetteville, 1979.

\bibitem {LZ71}W. A. J. Luxemburg and A. C. Zaanen, \textit{Riesz Spaces }I,
North-Holland, Amsterdam-London, 1971.

\bibitem {M91}P. Meyer-Nieberg, \textit{Banach Lattices}, Springer-Verlag,
Berlin-Heidelberg-New York, 1991.

\bibitem {P80}G. K. Pederson, \textit{Analysis Now}, Springer-Verlag,
Berlin-Heidelberg-New York, 1980.

\bibitem {S74}H. H. Schaefer, \textit{Banach Lattices and Positive Operators},
Springer-Verlag, Berlin-Heidelberg-New York, 1974.

\bibitem {S48}M. H. Stone, Notes on integration I, \textit{Proc. Nat. Acad.
Sci. USA}, 34 (1948), 336--342.

\bibitem {Z97}A. C. Zaanen, \textit{Introduction to Operator Theory in Riesz
Spaces}, Springer-Verlag, Berlin-Heidelberg-New York, 1997.
\end{thebibliography}
\end{document}